\RequirePackage{fix-cm}
\documentclass[smallextended]{svjour3}       % onecolumn (second format)
\smartqed  % flush right qed marks, e.g. at end of proof

\usepackage{amssymb}
\usepackage[english]{babel}
\usepackage{amsmath,mathtools}
\usepackage{microtype} 

\usepackage{tikz,tikz-cd}
\usetikzlibrary{patterns}

\newcommand{\dN}{\mathbb{N}}

\newcommand{\dR}{\mathbb{R}}

\newcommand{\calB}{\mathcal{B}}
\newcommand{\calC}{\mathcal{C}}

\newcommand{\calF}{\mathcal{F}}
\newcommand{\calG}{\mathcal{G}}
\newcommand{\calGM}{\mathcal{G}_{\mathrm{M}}}
\newcommand{\calGA}{\mathcal{G}_{\mathrm{A}}}
\newcommand{\calGS}{\mathcal{G}_{\mathrm{S}}}
\newcommand{\calGWS}{\mathcal{G}_{\mathrm{WS}}}
\newcommand{\calGTM}{\mathcal{G}_{\mathrm{TM}}}
\newcommand{\calGZM}{\mathcal{G}_{\mathrm{ZM}}}
\newcommand{\calGE}{\mathcal{G}_{\mathrm{E}}}
\newcommand{\calGTB}{\mathcal{G}_{\mathrm{TB}}}
\newcommand{\calGB}{\mathcal{G}_{\mathrm{B}}}

\newcommand{\calK}{\mathcal{K}}

\newcommand{\calN}{\mathcal{N}}

\newcommand{\calS}{\mathcal{S}}

\newcommand{\calW}{\mathcal{W}}

\newcommand{\frakP}{\mathfrak{P}}

\renewcommand{\phi}{\varphi}
\renewcommand{\epsilon}{\varepsilon}

\DeclareMathOperator{\conv}{conv}

\DeclareMathOperator{\selectope}{sel}
\DeclareMathOperator{\Min}{Min}

\begin{document}
\title{On Decomposition of Solutions for Coalitional Games\thanks{This research has been supported by the project RCI (CZ.02.1.01/0.0/0.0/16 019/0000765) and by the grant GA\v{C}R n. 16-12010S.}
}

%\titlerunning{Decomposition of Solutions}        % if too long for running head

\author{Tom\'{a}\v{s} Kroupa}


\institute{T. Kroupa \at Department of Computer Science, Faculty of Electrical Engineering Czech Technical University in Prague \email{tomas.kroupa@fel.cvut.cz} \and The Czech Academy of Sciences, Institute of Information Theory and Automation}

\date{Received: date / Accepted: date}

\maketitle

\begin{abstract}
A solution concept on a class of transferable utility coalitional games is a multifunction satisfying given criteria of economic rationality. Every solution associates a set of payoff allocations with a coalitional game. This general definition specializes to a number of well-known concepts such as the core, Shapley value, nucleolus etc. In this note it is shown that in many cases a solution factors through a set of games whose members can be viewed as elementary building blocks for the solution. Two factoring maps have a very simply structure. The first decomposes a game into its elementary components and the second one combines the output of the first map into the respective solution outcome. The decomposition is then studied mainly for certain polyhedral cones of zero-normalized games. 
\end{abstract}

\section{Introduction}
The main goal of this paper is to initiate research into the structure of certain solution concepts for coalitional games. Our setting is that of coalitional games with transferable utility and a finite player set. A solution on a class of games is a function that associates a set of feasible payoff allocations with a game from the class. See the book \cite{PelegSudholter07} for details on cooperative games and an in-depth exposition of key solution concepts. 

Section \ref{sec:coalgames} formally introduces basic notions and elementary results about coalitional games and their special families. In particular, it is emphasized that zero-normalization enables us to transform certain polyhedral cones of games into pointed polyhedral cones, which contain extreme rays and are thus conic hulls of their finite sets of generators.

The decomposition for solutions is defined in Section \ref{sec:dec}. The main idea is captured by the commutative diagram \eqref{diagram}, which is often repeated in cases of particular solution concepts for the sake of clarity. It is then shown that several important solutions allow for such a factorization. Specifically, in the rest of Section \ref{sec:dec} it is proved that every probabilistic value, nucleolus, Weber set, and selectope decompose in the sense of diagram \eqref{diagram}. There are other solutions, however, for which the sought decomposition is not natural or it may not exist at all. As examples of such solutions we can mention Von Neumann--Morgenstern stable sets or bargaining sets. 

Interestingly enough, the core solution factorizes when it is restricted onto a suitable cone of games. This is explained in Section \ref{sec:add} whose results apply to additive solutions defined on the linear space of all games or on the cones of games containing additive games as the lineality space. Such cones of games are, for example, the cone of supermodular games, exact games, and totally balanced games, respectively. The existence of decomposition then amounts to the statement that the corresponding solutions are completely determined by the solution map restricted to a finite set of generators; see Proposition \ref{pro:lin} and Theorem \ref{thm:cone}. For example, it is proved in \cite{StudenyKroupa16}  that the cores of supermodular games coincide with the class of convex polytopes known as generalized permutohedra, and the cores associated with extreme supermodular games are precisely the so-called indecomposable generalized permutohedra. Recent progresses in the study of the cone of balanced games make it possible to frame an analogous question for the generators of the ``maximal linear regions" where the core is additive and positively homogeneous; see \cite{Vermeulen18} for details. 
Whereas the core solution is additive when restricted to particular classes of games, there exists its decomposition \eqref{diagram} using lattice operations only; see Section \ref{subsec}. Inspired by the max-convex representation of coalitional games \cite{llerenaRafels06}, we show that the cores of weakly superadditive games are fairly special convex polytopes; this is a consequence of Theorem \ref{thm:WSmain}.

\section{Coalitional Games} \label{sec:coalgames}

We use the standard notions and results from cooperative game theory; see \cite{PelegSudholter07}. Let $N= \{1,\dots,n\}$ be a finite set of \emph{players} for some integer $n\geq 2$ and let $\frakP(N)$ be the powerset of $N$. Any set  $A\in \frakP(N)$ is called a \emph{coalition}. A~\emph{(transferable-utility coalitional) game} is a  function $v\colon \frakP(N)\to \dR$ satisfying $v(\emptyset)=0$. By $\calG(N)$ we denote the linear space of all games with the player set $N$. The dimension of $\calG(N)$ is $2^n-1$. For any nonempty coalition $A\subseteq N$, the restriction of $v\in\calG(N)$ to $\frakP(A)$ is called a \emph{subgame} of $v$.
 A game $v\in\calG(N)$ is called 
\begin{itemize}
\item \emph{weakly superadditive} whenever 
$v(A\cup \{i\})\geq v(A)+v(\{i\})$ holds for all $A\subseteq N$ and every $i\in N\setminus A$,
\item \emph{monotone} if $v(A)\leq v(B)$ for all  $A\subseteq B\subseteq N$,
\item \emph{supermodular} if $v(A\cup B)+v(A\cap B)\geq v(A)+v(B)$ for all $A,B\subseteq N$,
\item \emph{totally monotone} if for any $k\geq 2$ and all $A_1,\dots,A_k\subseteq N$,
\[
v(\bigcup_{i=1}^k A_i) \geq \sum_{\substack{I \subseteq \{1,\dots,k\}\\ I\neq \emptyset}} (-1)^{|I|+1}\cdot  v(\bigcap_{i\in I}A_i),
\]
\item \emph{zero-normalized} if $v(\{i\})=0$ for all $i\in N$,
\item \emph{zero-monotone} if it is zero-normalized and monotone,
\item \emph{additive} if $v(A\cup B)=v(A)+v(B)$ for all $A,B\subseteq N$ with $A\cap B=\emptyset$.
\end{itemize}
By $\mathsf{0}$ we denote the game in $\calG(N)$ that is identically equal to $0$.
For any nonempty coalition $A\subseteq N$, a \emph{unanimity game (on $A$)} is given   by
\[
u_A(B)= \begin{cases}
1 & A\subseteq B, \\
0 & \text{otherwise,} \end{cases} \qquad B\subseteq N.
\]
The following notation will be used for particular sets of games:
\begin{align*}
\calGWS(N) & = \{v\in \calG(N) \mid \text{$v$ is weakly superadditive}\} \\
\calGM(N) & = \{v\in \calG(N) \mid \text{$v$ is monotone}\} \\
\calGS(N) & = \{v\in \calG(N) \mid \text{$v$ is supermodular}\} \\
\calGTM(N) & = \{v\in \calG(N) \mid \text{$v$ is totally monotone}\} \\
\calGZM(N) & = \{v\in \calG(N) \mid \text{$v$ is zero-monotone}\} \\
\calGA(N) &=  \{v\in \calG(N) \mid \text{$v$ is additive}\} \\
\calG^0(N) & = \{v\in \calG(N) \mid \text{$v$ is zero-normalized}\} \\
\calG_{\star}^0(N) & = \calG_{\star}(N)\cap \calG^0(N), \text{where $\star\in \{\mathrm{WS},\dots\}$}  
\end{align*}
Many of the sets of games above are in fact polyhedral cones. For all the unexplained notions concerning convexity and polyhedral sets we refer the reader to \cite{BachemKern92,Ewald96}. Let $\star\in \{\mathrm{WS},\mathrm{S},\mathrm{TM}\}$. 
It follows immediately from the definitions that $\calG_{\star}(N)$ is a polyhedral cone, $\calGA(N)$ is a linear space, and $\calGA(N) \subseteq \calG_{\star}(N)$. 
  Since $\calGA(N)=\calG_{\star}(N)\cap -\calG_{\star}(N)$, the linear space $\calGA(N)$ is the lineality space of $\calG_{\star}(N)$. Moreover, the linear space $\calG^0(N)$ is complementary to $\calGA(N)$ in $\calG(N)$, that is, $\calG^0(N)\cap \calGA(N)=\{\mathsf{0}\}$ and $\calG^0(N) + \calGA(N)=\calG(N)$. Thus, for any game $v\in \calG(N)$, there exist unique $w\in \calG^0(N)$ and $m\in \calGA(N)$ such that $v=w+m$. Indeed, it suffices to put $w= \hat{v}$ and $m= v-\hat{v}$, where 
\begin{equation} \label{eq:zeronorm}
\hat{v}(A) = v(A) - \sum_{i\in A}v(\{i\}), \quad A\subseteq N,
\end{equation}
and observe that $\hat{v}\in \calG^0(N)$ and $v-\hat{v}\in \calGA(N)$.

It follows from the above considerations that $\calG_{\star}^0(N)$
 is a pointed polyhedral cone and we obtain the decomposition
 \begin{equation} \label{eq:dirdecomp}
\calG_{\star}(N)=\calG_{\star}^0(N) + \calGA(N), \qquad \star\in \{\mathrm{WS},\mathrm{S},\mathrm{TM}\}.
\end{equation}
This means that for any game $v\in\calG_{\star}(N)$ there exist unique $w\in \calG_{\star}^0(N)$ and $m\in \calGA(N)$ such that $v=w+m$, where necessarily $w=\hat{v}$ and $m=v-\hat{v}$.

\section{Decomposition of Solutions}\label{sec:dec}
A vector $x = (x_1,\dots,x_n)\in \dR^n$ is called a \emph{payoff allocation}. The total payoff allocation assigned to a nonempty coalition $A\subseteq N$ is the number $x(A)= \sum_{i\in A}x_i$ and we further define $x(\emptyset)= 0$. There is an obvious linear isomorphism
\begin{equation}\label{liniso}
e\colon \dR^n\to  \calGA(N)
\end{equation}
such that  $e(x)= m_x\in\calGA(N)$, where $m_x$ is the additive game given by
\begin{equation}\label{def:add}
m_x (A) = x(A),\qquad A\subseteq N.
\end{equation}
Note that  the inverse linear mapping $e^{-1}\colon \calGA(N)\to \dR^n$ just restricts an additive game $v\in \calGA(N)$ to the atoms of $\frakP(N)$, 
\begin{equation}\label{def:addinv}
e^{-1}(v) = (v(\{i\}))_{i\in N}.
\end{equation}

We define solutions of coalitional games. Let $\frakP(\dR^n)$ be the powerset of $\dR^n$.   By $\Gamma(N)$ we denote an arbitrary nonempty subset of $\calG(N)$. A \emph{solution} is a set-valued mapping 
\[
\sigma\colon \Gamma(N)\to \frakP(\dR^n).
\]
Thus, every element $x\in \sigma(v)$, where $v\in\Gamma(N)$, is considered as a final payoff allocation in the game $v$. We remark that every solution $\sigma$ according to \cite[Definition 2.3.1]{PelegSudholter07} must  also satisfy \emph{feasibility}, that is,
\begin{equation}\label{cond:fea}
 \sigma(v)\subseteq \{x\in\dR^n \mid x(N)\leq v(N)\}, \qquad v\in \Gamma(N).
 \end{equation}
Most of the solutions discussed in further sections of this paper meet the condition \eqref{cond:fea}.
Additional assumptions on $\sigma$ determine various solution concepts, such as the core mapping, selectope, nucleolus, Weber set etc. By $X + Y$ we denote the \emph{Minkowski sum} of sets $X$ and $Y$  in $\dR^n$ defined by
\[
X+Y = \{x+y \mid x\in X, y\in X\}.
\]
For any $c\in\dR$, let $c\cdot X= \{cx \mid x\in X\}$. When $X$ is convex, there is no ambiguity in writing $cX$,  where $c\in\dN$, since in this case we have
\[
c\cdot X=\underbrace{X+\dots+X}_{cX}.
\]

 We say that a solution $\sigma\colon \Gamma(N)\to \frakP(\dR^n)$ is 
\begin{itemize}
\item \emph{nonempty} if $\sigma(v)\neq \emptyset$ for all $v\in \Gamma(N)$,
\item \emph{single-valued} if $\sigma(v)$ is a singleton for all $v\in \Gamma(N)$,
\item \emph{superadditive} if $\sigma(v+w)\supseteq \sigma(v) + \sigma(w)$ for all $v,w\in \Gamma(N)$,
\item \emph{additive} if $\sigma(v+w)=\sigma(v) + \sigma(w)$ for all $v,w\in \Gamma(N)$,
\item \emph{covariant under strategic equivalence} if, for every $v\in\Gamma(N)$, $c>0$, and every $w\in\calGA(N)$, we have $\sigma(cv+w)=c\cdot \sigma(v) + \{e^{-1}(w)\}$, where $e^{-1}$ is as in \eqref{def:addinv}, 
\item \emph{positively homogeneous} if $\sigma(cv)=c\cdot \sigma(v)$ for all $v\in \Gamma(N)$ and every $c>0$,
\item \emph{positive} if, for every nonnegative game $v\in\calGTM(N)$, any $x\in \sigma(v)$ has only nonnegative coordinates.
\end{itemize}
\begin{remark}
Let $\calGTM^+(N)$ be the family of all nonnegative totally monotone games. Note that 
\[
\calGTM^+(N)=\calGTM(N) \cap \calGM(N).
\]
Then $\calGTM^+(N)$ forms a pointed polyhedral cone, which plays a crucial role for the representation of games and the respective Choquet integrals; see \cite{GilboaSchmeidler94}. In particular, letting $\calGTM^+(N)$ be the positive cone of $\calG(N)$, we obtain a partial order $\leqslant$ on $\calG(N)$ such that $v\leqslant w$ whenever $w-v\in\calGTM^+(N)$. Thus, $\calG(N)$ becomes a Riesz space whose order is $\leqslant$, and for any game $v\in\calG(N)$ there exist uniquely determined games $w_1,w_2\in\calGTM^+(N)$ such that $v=w_1-w_2$. With these conventions in mind, the definition of a positive solution can be rephrased as follows: For any $\mathsf{0}\leqslant  v\in \calG(N)$, every vector $x\in \sigma(v)$ satisfies $0\leq x$, where $\leq$ is the pointwise order on $\dR^n$.

\end{remark}

The main goal of this paper is to discuss the cases when a given solution $\sigma\colon \Gamma(N)\to \dR^n$ can be decomposed as follows. 
\begin{equation}\label{diagram}
   \begin{tikzcd}
   \Gamma(N) \arrow[rd,"\sigma"] \arrow[r, "\tau"] & \Omega(N)^Z \arrow[d,"\alpha"] \\
   & \frakP(\dR^n)
     \end{tikzcd}
\end{equation}
In the above diagram $Z$ is a nonempty set and $\emptyset\neq \Omega(N)\subseteq \calG(N)$. Since \eqref{diagram} commutes, $\sigma=\alpha \circ \tau$, where $\tau\colon \Gamma(N)\to \Omega(N)^Z$ and $\alpha \colon \Omega(N)^Z\to \frakP(\dR^n)$ are some maps. The interpretation is that the solution $\sigma$ factors through the set $\Omega(N)^Z$, whose elements  are maps $\omega\colon Z\to \Omega(N)$. We will use the notation $\omega_z$ to denote the image of $z\in Z$ under $\omega$. Hence, the map $\omega$ selects a game $\omega_z\in  \Omega(N)$, for every $z\in Z$. The point is that the games in $\Omega(N)$ can be considered as elementary building blocks for the construction of the solution $\sigma$. The map $\alpha$ then aggregates the games $\omega_z$ into a set of payoff allocations $\alpha(\omega)=\alpha(\tau(v))=\sigma(v)$, for every $v\in \Gamma(N)$. The dependence of $Z$, $\Omega(N)$, $\tau$, and $\alpha$ on the solution mapping $\sigma$ is tacitly understood.

In the sequel we are interested only in the decompositions of $\sigma$ that are non-trivial in the following sense. Put $\Omega(N)= \Gamma(N)$ and  $Z= \{z\}$. Then the set of all maps $Z\to \Omega(N)$ is in bijection with $\Omega(N)$, so it is sensible to identify a map $\omega\colon Z\to \Omega(N)$ with the unique element $v\in \Omega(N)$ in the range of $\omega$. Hence, we can define $\tau(v)= v$ and $\alpha(v)= \sigma(v)$, which makes the diagram \eqref{diagram} commute trivially. In all other cases we will say that $\sigma$ has a \emph{non-trivial decomposition}.

In the rest of this section we will show that many existing solution concepts have a non-trivial decomposition \eqref{diagram}. Moreover, such decompositions are often easily constructed from the corresponding definitions of those solutions.

\subsection{Probabilistic values}
A mapping $\psi=(\psi_1,\dots,\psi_n)\colon \calG(N)\to\dR^n$ is a \emph{probabilistic value} \cite{Weber88} if for each player $i\in N$ there exists a probability measure $p_i\colon \frakP(N\setminus \{i\})\to [0,1]$ such that
\[
\psi_i(v) = \sum_{A\subseteq N\setminus\{i\}} p_i(A)\cdot (v(A\cup \{i\})-v(A)), \qquad v\in\calG(N).
\]
For all $v\in\calG(N)$, we will denote the marginal contribution of player $i\in N$ to an arbitrary coalition $A\subseteq N$ by $D^v_i(A)= v(A\cup \{i\})-v(A)$. Note that $A$ is allowed to contain the player $i$. This yields a game $D_i^v\in\calG(N)$ and a map $D^v\colon i\in N\mapsto D_i^v$.
In order to  see that the solution $\psi$ decomposes non-trivially, it suffices to put $\Gamma(N)=\Omega(N)= \calG(N)$, $Z= N$, and $\tau(v)= D^v$ for all $v\in\calG(N)$. Let $\alpha\colon \calG(N)^N\to \dR^n$ be defined by
\[
\alpha_i(\omega) =  \sum_{A\subseteq N\setminus\{i\}} p_i(A) \cdot \omega_i(A),\qquad \omega=(\omega_j)_{j\in N}\in \calG(N)^N, \,i\in N.
\]
Then the diagram \eqref{diagram} commutes with these definitions, since for all $v\in\calG(N)$ and every $i\in N$,
\[
\alpha_i(\tau(v))=\alpha_i(D^v)= \sum_{A\subseteq N\setminus\{i\}} p_i(A) \cdot D^v_i(A) =\psi_i(v).
\]

\subsection{Nucleolus}
Let $K$ be a nonempty compact convex subset of $\dR^n$.
For any $x\in K$ and every $v\in\calG(N)$, let $\theta^x_v$ be a game such that $\theta^x_v(A)= v(A)-x(A)$, where $A\subseteq N$. Further, define a vector $\theta(x)= (\theta_{i}(x))_{i=1,\dots,2^n}$ with coordinates $\theta_{i}(x)= \theta^x_v(A_i)$, where $A_1,\dotsc,A_{2^n}$ is an enumeration of all the subsets of $N$ satisfying the condition:
\[
\text{For all $i,j 
=1,\dots,2^n$, if $i<j$, then $\theta^x_v(A_i)\geq \theta^x_v(A_j)$.}
\]
Let $\preceq$ denote the lexicographic order on $\dR^{2^n}$. The \emph{nucleolus} of $v$ with respect to $K$ is the set
\[
\calN(v,K) = \{x\in K \mid \theta(x)\preceq \theta(y), \text{ for all $y\in K$}\}.
\]
Since $K\neq \emptyset$ is compact convex, the nucleolus $\calN(v,K)$ is nonempty and single-valued for every $v\in\calG(N)$; see  \cite{SchmeidlerNucl69} for further details.

We will show that \eqref{diagram} commutes with the data defined as follows. Let $\Gamma(N)=\Omega(N)=\calG(N)$, $Z= K$. For any $v\in \calG(N)$, let $\theta_v$ be the mapping $x\in K\mapsto \theta_v^x\in\calG(N)$ and put
\[
\tau(v)= \theta_v, \qquad v\in\calG(N).
\]
Let $\omega$ be any mapping $K\to \calG(N)$ and the image of $x\in K$ under $\omega$ be now denoted by $\omega^x$. We define a vector $\omega(x)\in \dR^{2^n}$ using $\omega^x$ completely analogously to the definition of $\theta(x)$ using $\theta^x_v$ above. Let $\Min S$ be the lexicographic minimum of a  set $\emptyset \neq S\subseteq \dR^{2^n}$. Put
\[
\alpha(\omega)= \Min \{\omega(x) \mid x\in K \}, \qquad \omega \in \calG(N)^K.
\]
Thus, for every $v\in \calG(N)$,
\[
\alpha(\tau(v)) = \alpha(\theta_v)= \Min \{\theta(x) \mid x\in K\}=\calN(v,K),
\]
which means that \eqref{diagram} is commutative using the definitions above.

\subsection{Weber set}
Let $\Pi(N)$ denote the set of all permutations $\pi\colon N\to N$.
A~\emph{marginal vector} of a game $v\in\calG(N)$ with respect to  $\pi\in\Pi(N)$ is the payoff allocation $x^{v,\pi}\in\dR^n$ with coordinates
\begin{equation}\label{def:MargVec}
x^{v,\pi}_i= v\left(\bigcup_{j\leq \pi^{-1}(i)}\{\pi(j)\}\right)-v\left(\bigcup_{j< \pi^{-1}(i)}\{\pi(j)\}\right), \qquad i\in N.
\end{equation}
The \emph{Weber set} of a game $v\in\calG(N)$ is the convex hull of all the marginal vectors of $v$,
\[
\calW(v)=\conv \{x^{v,\pi}\mid \pi \in \Pi(N)\}.
\]

Let $x^v\colon \Pi(N)\to \dR^n$ be the \emph{payoff-array transformation}, which was considered in \cite{StudenyKroupa16}. Precisely, $x^v$ is a mapping $\pi \in \Pi(N)\mapsto x^{v,\pi}\in \dR^n$. Put $\Gamma(N)= \calG(N)$, $\Omega(N)= \calGA(N)$, $Z= \Pi(N)$, and $\tau(v)= e\circ x^v$ for every $v\in \calG(N)$, where $e$ is as in \eqref{liniso}. Define  
\[
\alpha(\omega) = \conv \{e^{-1}(\omega_{\pi})\mid \pi \in\Pi(N) \},\qquad \omega\in \calGA(N)^{\Pi(N)},
\]
and observe that the diagram \eqref{diagram:Weber} commutes with these definitions since
\[
\alpha(\tau(v))=\alpha(e\circ x^v)=\conv \{e^{-1}(e(x^v_{\pi})) \mid \pi\in \Pi(N)\}=\calW(v), \qquad v\in\calG(N).
\]
\begin{equation}\label{diagram:Weber}
   \begin{tikzcd}
   \calG(N) \arrow[rd,"\calW"] \arrow[r, "\tau"] & \calGA(N)^{\Pi(N)} \arrow[d,"\alpha"] \\
   & \frakP(\dR^n)
     \end{tikzcd}
\end{equation}

\subsection{Selectope}
The selectope contains all possible reasonable distributions of Harsanyi dividends among the players; see \cite{DerksHallerPeters00}, for example. Specifically, it is constructed as follows. A mapping $a\colon \frakP(N)\setminus\{\emptyset\}\to N$ such that $a(A)\in A$ is called a \emph{selector}. Let $\calS(N)$ be the set of all selectors. The \emph{selector value} of a game $v\in\calG(N)$ corresponding to $a\in\calS(N)$ is a vector $\phi^{a}(v)\in\dR^n$ with coordinates
\[
\phi_i^{a}(v)= \sum_{\substack{A\in\frakP(N)\setminus\{\emptyset\} \\ a(A)=i}} m^v(A), \qquad i\in N,
\]
where $m^v$ is the M\"obius transform (Harsanyi dividend) of $v$ given by
\[
m^v(A)=\sum_{B\subseteq A} (-1)^{|A\setminus B|} \cdot v(B), \qquad A\subseteq N.
\]
The \emph{selectope} of $v$ is then the set
\[
\selectope(v) = \conv \{\phi^{a}(v)\mid a \in \calS(N)\}.
\]
It is clear that $\selectope(v)\neq \emptyset$ for all games $v\in \calG(N)$.

The definition of solution $\selectope$ on $\calG(N)$ is in fact captured by the diagram \eqref{diagram}. Indeed, put $\Gamma(N)= \calG(N)$, $\Omega(N)= \calGA(N)$, and $Z= \calS(N)$. For any game $v\in \calG(N)$,  let $\phi(v)\colon \calS(N)\to \dR^n$ be the map $a\in\calS(N)\mapsto \phi^{a}(v)$. Define the map
\[
\tau\colon \calG(N) \to \calGA(N)^{\calS(N)} 
\]
 as $\tau(v)= e\circ \phi(v)$, where $e$ is defined by \eqref{liniso}.
Further, let $\omega\in \calGA(N)^{\calS(N)}$ and put
\[
\alpha(\omega) = \conv\{e^{-1}(\omega_a) \mid a\in\calS(N)\},
\]
where $\omega_a\in \calGA(N)$. Thus, for any $v\in\calG(N)$,
\[
\alpha(\tau(v)) =\alpha(e\circ \phi(v))=\conv\{\phi^{a}(v)  \mid a\in\calS(N)\} 
 = \selectope(v).
\]
This means that the diagram \eqref{diagram} commutes with the above definitions.

\section{Decomposition of Additive Solutions}\label{sec:add}
We will first look at a special case when the domain of a solution $\sigma$ is even a linear space. The typical example of a non-trivial factorization \eqref{diagram} of $\sigma$ is when $\sigma$ is an additive solution satisfying additional conditions on the linear space of all games $\calG(N)$.
\begin{proposition} \label{pro:lin}
Let $\calB= \{v_1,\dots,v_k\}$ be a basis of $\calG(N)$ and $\sigma$ be a nonempty solution on $\calG(N)$ that is superadditive, positive, and  such that $\sigma(\mathsf{0})=\{0\}$.

Put $Z= \{1,\dots,k\}$ and $\bar{\calB}= \{c_iv_i \mid v_i \in \calB, c_i\in \dR,i\in Z\}$.  For any $v\in\calG(N)$, define $\tau(v)= (a_i v_i)_{i\in Z}$, where $\sum_{i\in Z}a_i v_i$ is the unique linear combination expressing $v$, and

\[
\alpha(\omega) = \sum_{i\in Z}c_i\cdot \sigma(v_i), \qquad \omega=(c_i v_i)_{i\in Z},\, c_i\in \dR.
\]
Then this diagram commutes:
\begin{equation}\label{diagram:pro1}
   \begin{tikzcd}
   \calG(N) \arrow[rd,"\sigma"] \arrow[r, "\tau"] & \bar{\calB}^{Z} \arrow[d,"\alpha"] \\
   & \frakP(\dR^n)
     \end{tikzcd}
\end{equation}

\end{proposition}
\begin{proof}
It is clear that the definition of $\tau$ is correct since $\calB$ is a basis. 
First, we will prove that $\sigma$ is a single-valued solution that is even a linear mapping $\calG(N)\to \dR^n$. It follows from the assumptions about $\sigma$ that
\[
\{0\}=\sigma(\mathsf{0}) \supseteq \sigma(v) + \sigma(-v), \qquad \text{for all $v\in\calG(N)$.}
\]
Since $\sigma(v)\neq \emptyset$, this implies that $\sigma(v)$ is necessarily a singleton. Then superadditivity says that, for all $v,w\in\calG(N)$ and some $x,y,z\in\dR^n$,
\[
\{x\}=\sigma(v+w) \supseteq \sigma(v)+ \sigma(w)=\{y\} + \{z\}=\{y+z\},
\]
 which means that $\sigma$ is an additive mapping $\calG(N)\to\dR^n$. Observe that additivity implies $\sigma(-v)=-\sigma(v)$. We will show that $\sigma$ is even a linear mapping. First, let $a$ and $b$ be integers with $b\neq 0$. Then
\[
 \sigma (\tfrac{a}{b}v)=a\sigma (\tfrac{1}{b}v)=\tfrac{a}{b}b\sigma (\tfrac{1}{b}v)=\tfrac{a}{b}\sigma(v).
\] 
Hence, $\sigma$ is a linear mapping when $\calG(N)$ is understood as a vector space over the field of rational numbers. Assume now that $v$ is a nonnegative totally monotone game. Let $a\in\dR$ and $p,q$ be rational numbers such that $p\leq a \leq q$. Since $v$ is nonnegative, we get $pv\leq av \leq qv$. Hence,
\[
p\sigma(v) = \sigma(pv) \leq \sigma(av) \leq \sigma(qv)=q\sigma(v).
\]
Since $\sigma$ is a positive mapping, we get $\sigma(av)=a\sigma(v)$. As every game in $\calG(N)$ is a difference of nonnegative totally monotone games, it follows that $\sigma$ is a linear mapping.

Now, let $v=\sum_{i\in Z}a_i v_i$ for necessarily unique $a_i\in \dR$ and $v_i\in \calB$, where $i\in Z$. Then, by linearity of $\sigma$,
\[
\alpha(\tau(v))=\alpha((a_i v_i)_{i\in Z})= \sum_{i\in Z}a_i\cdot \sigma(v_i)= \sum_{i\in Z} \sigma(a_iv_i)=\sigma(\sum_{i\in Z}a_i v_i)=\sigma(v).
\]
Hence, the diagram \eqref{diagram:pro1} commutes. \qed
\end{proof}

The seemingly weak conditions of Proposition \ref{pro:lin}  make $\sigma$ into a linear map. On top of that, since its domain is the entire linear space $\calG(N)$, the solution $\sigma$ is fully determined by the images $\sigma(v_i)$ of the basis elements $v_i\in\calB$. We will now consider a more natural situation when $\sigma$ is not considered on $\calG(N)$, but rather on a smaller set of games such as the polyhedral cones of games discussed in Section \ref{sec:coalgames}. Other examples of such classes of games include exact games and totally balanced games whose definititions are repeated below; see \cite{Schmeidler72} and \cite{KalaiZemel1982}, respectively.

For any game $v\in \calG(N)$, the set $\calC(v)$ of all allocations that are Pareto efficient and coalitionally rational is called the \emph{core} of $v$. Precisely,
\[
\calC(v)= \{x\in \dR^n \mid \text{$x(N)=v(N)$ and $x(A)\geq v(A)$ for all $A\subseteq N$}\}.
\]
A game $v\in\calG(N)$ is called
\begin{itemize}
\item \emph{balanced} if $\calC(v)\neq \emptyset$,
\item \emph{exact} if $v(A)=\min \{x(A) \mid x\in \calC(v)\}$ for all $A\subseteq N$,
\item \emph{totally balanced} if every subgame of $v$ is balanced.
\end{itemize}
Every exact game is totally balanced.
Put
\begin{align*}
\calGB(N) & = \{v\in \calG(N) \mid \text{$v$ is balanced}\} \\
\calGE(N) & = \{v\in \calG(N) \mid \text{$v$ is exact}\} \\
\calGTB(N) & = \{v\in \calG(N) \mid \text{$v$ is totally balanced}\} \\
\calG_{\star}^0(N) & = \calG_{\star}(N)\cap \calG^0(N), \text{where $\star\in \{\mathrm{B},\mathrm{E},\mathrm{TB}\}$}  
\end{align*}
It is wellknown that
\[
\calGTM(N)\subseteq \calGS(N)\subseteq \calGE(N) \subseteq \calGTB(N) \subseteq \calGB(N),
\]
where all the inclusions are proper for $n\geq 4$. The set of all balanced games $\calGB(N)$ is a polyhedral cone as a direct consequence of the Bondareva--Shapley theorem. The convex cones $\calGE(N)$ and $\calGTB(N)$ can be described by finitely-many linear inequalities too; see \cite{Csoka11balancedness} and \cite{Lohmann2012minimal}. None of those cones is pointed, however, since each of them contains the set of all additive games $\calGA(N)$ as the lineality space. Then the same technique as in Section \ref{sec:coalgames} can be employed to show that 
\begin{equation}\label{eq:decompcone}
\calG_{\star}(N)=\calG_{\star}^0(N) + \calGA(N),  
\end{equation}
where $\star\in \{\mathrm{B},\mathrm{E},\mathrm{TB}\}$ and $\calG_{\star}^0(N)$ is a pointed polyhedral cone. 

Minkowski's theorem says that $\calG_{\star}^0(N)$ is the conic hull of its (finitely-many) extreme rays. Needless to say, the expression of a given game as a conic combination of generators for the cone is usually highly non-unique. We will need the following result (Lemma \ref{lem:subdivision}), which makes it possible to achieve uniqueness of the conic representation with respect to a chosen subdivision of the cone into simplex cones; see \cite[Theorem III.1.12]{Ewald96}. We recall the needed terminology. A pointed polyhedral cone $C$ in a finite-dimensional vector space is a \emph{simplex cone} whenever the finite set of generators of its extreme rays is linearly independent. A \emph{polyhedral fan} is a finite set $\calF$ of polyhedral cones  satisfying the following conditions:
\begin{enumerate}
\item If $C\in\calF$ and $F$ is a face of $C$, then $F\in \calF$.
\item If $C,D\in\calF$, then the intersection $C\cap D$ is a face of both $C$ and $D$.
\end{enumerate}
A \emph{simplicial fan} is a polyhedral fan whose every cone is a simplex cone. The set $\bigcup_{C\in \calF} C$ is the \emph{support} of a polyhedral fan $\calF$.

\begin{lemma} \label{lem:subdivision}
 Let $C$ be a pointed polyhedral cone. Then there exists a simplicial fan $\calF= \{C_1,\dots,C_k\}$ such that:
\begin{enumerate}
\item The support of $\calF$ is $C$.
\item For any $i\in \{1,\dots,k\}$, there is a subset $S_i$ of the generators for $C$ such that $C_i$ is the conic hull of $S_i$.
\end{enumerate}
Moreover, for every $x\in C$, there exists a unique simplex cone $C_i\in \calF$ with $x\in C_i$ and such that $C_i$ has the smallest dimension among the cones $C_j\in \calF$ containing $x$. 
\end{lemma}

The import of Lemma \ref{lem:subdivision} is that it enables us to express a given element in a cone uniquely, although the uniqueness is always understood with respect to an arbitrarily chosen simplicial fan satisfying the conditions of Lemma \ref{lem:subdivision}. Namely any $x\in C$ has the unique expression as a conic combination of the generators for the smallest simplex cone $C_i\in \calF$ with $x\in C_i$, where $\calF$ is any simplicial fan subdividing $C$ in the sense of Lemma \ref{lem:subdivision}.

Recall that by $\hat{v}$ we denote the zero-normalized game \eqref{eq:zeronorm}. In the next result the set of games $\calK(N)$ is any polyhedral cone satisfying the conditions of Theorem \ref{thm:cone}, not necessarily any of the cones $\calG_{\star}(N)$ discussed above. Denote $\calK^0(N)=\calK(N)\cap\calG^0(N)$.
\begin{theorem} \label{thm:cone}
Let $\calK(N)\subseteq \calG(N)$ be a polyhedral cone whose lineality space is $\calG_A(N)$ and assume that  $\sigma$ is an additive and positively homogeneous solution on $\calK(N)$. Let $\calB= \{v_1,\dots,v_k\}$ be a set of generators of $\calK^0(N)$ and put  $Z= \{1,\dots,k,k+1\}$, $\bar{\calB}= \{c_i v_i \mid v_i \in \calB, c_i\geq 0,i=1,\dots,k\}$. Further, let $\calF$ be a simplicial fan whose support is $\calK^0(N)$.

  For any $v\in\calK(N)$, let $\tau(v)= (a_1v_1,\dots,a_kv_k,v-\hat{v})$, where $\sum_{i=1}^k a_i v_i$ is the unique conic combination expressing $\hat{v}$ with respect to $\calF$. Put
\[
\alpha(\omega) = \sum_{i\in Z}c_i\cdot \sigma(w_i), \qquad \omega=(c_i w_i)_{i\in Z}\in (\bar{\calB}\cup \calG_A(N))^Z, \, c_i\geq 0.
\]
Then this diagram commutes:
\begin{equation}\label{diagram:pro2}
   \begin{tikzcd}
  \calK(N) \arrow[rd,"\sigma"] \arrow[r, "\tau"] & (\bar{\calB}\cup \calG_A(N))^Z \arrow[d,"\alpha"] \\
   & \frakP(\dR^n)
     \end{tikzcd}
\end{equation}
\end{theorem}
\begin{proof}
The definition of $\tau$ is sensible by Lemma \ref{lem:subdivision}. We need to verify that \eqref{diagram:pro2} is commutative. Let $v\in \calK(N)$. Then
\[
\alpha(\tau(v))=\alpha((a_1v_1,\dots,a_kv_k,v-\hat{v}))=\sum_{i=1}^k a_i\cdot \sigma(v_i) + \sigma(v-\hat{v}).
\]
By additivity and positive homogeneity of $\sigma$,
\[
\sum_{i=1}^k a_i\cdot \sigma(v_i) = \sigma (\sum_{i=1}^k a_i v_i)=\sigma(\hat{v}).
\]
Thus, $\alpha(\tau(v))=\sigma(\hat{v})+\sigma(v-\hat{v})=\sigma(v)$.
\qed
\end{proof}
It is sensible to apply Theorem \ref{thm:cone} to any cone $\calK(N)$ that satisfies the conditions above and, at the same time, whose structure of extreme rays is known or for which extremality of a given game in the cone is not too difficult to check. An example thereof is the cone of supermodular games \cite{StudenyKroupa16} and there are strong indications that also the cone of exact games is amenable to such a description; see \cite{KroupaStudeny18} for the details.

\subsection{Nonadditive decomposition of core}\label{subsec}
In this section the linear space of all games $\calG(N)$ is considered with the lattice order given by the pointwise supremum $\vee$ and the pointwise infimum $\wedge$. We will make use of the following nonadditive representation of the core solution, which is mentioned in \cite{llerenaRafels06}.
\begin{lemma}\label{pro:coreint}
If $v_1,\dots,v_k\in \calG(N)$ are games satisfying the condition $v_1(N)=\dots= v_k(N)$, then $\calC(\bigvee_{i=1}^k v_i)=\bigcap_{i=1}^k \calC(v_i)$.
\end{lemma}
\begin{proof}
Put $v= \bigvee_{i=1}^k v_i$. Let $x\in \calC(v)$. For every $i=1,\dots,k$ and all $A\subseteq N$, the conditions  $x(N)=v(N)=v_i(N)$ and $x(A) \geq v(A)\geq v_i(A)$ are satisfied, which gives $x\in \bigcap_{i=1}^k \calC(v_i)$. 
Conversely, assume that $x\in \bigcap_{i=1}^k \calC(v_i)$. Then $x(N)=v_i(N)=v(N)$ for any $i=1,\dots,k$. Let $A\subseteq N$. Since $x(A)\geq v_i(A)$ for all $i=1,\dots,k$ and $v$ is the supremum of games $v_1,\dots,v_k$, it follows that $x(A)\geq v(A)$. Hence, $x\in \calC(v)$. \qed
\end{proof}

From now one we focus on the cone of weakly superadditive games $\calG_{\mathrm{WS}}(N)$. Clearly, $\calG_{\mathrm{WS}}^0(N)=\calGZM(N)$ and then \eqref{eq:dirdecomp} gives
\[
\calG_{\mathrm{WS}}(N)=\calGZM(N) + \calGA(N).
\]
We will base our representation on the polyhedral cone $\calGTM^0(N)$ of zero-normalized totally monotone capacities. Observe that every game $v\in \calGTM^0(N)$ is even monotone. Indeed, let $A\subseteq B\subseteq N$. It is enough to prove $v(A)\leq v(B)$ in case that $|B|=|A|+1$. Let $B=A\cup \{i\}$, where $i\notin N\setminus A$. Then
\[
v(B)=v(A\cup \{i\}) \geq v(A)+v(\{i\})-v(\emptyset)=v(A).
\]
Thus, $\calGTM^0(N)\subseteq \calGZM(N)$.
It is easy to show that the polyhedral cone $\calGTM^0(N)$ is pointed and the generators of its extreme rays are the unanimity games $u_B$ for all $B\subseteq N$ such that $|B|\geq 2$ (see, e.g., \cite[Theorem 2.58]{Grabisch16}). Thus, every game $v\in \calGTM^0(N)$ is \emph{almost positive}, which means that there exist $\lambda_A\geq 0$, where $A \subseteq N$ and $|A| \geq 2$, such that
\[
v=\sum_{\substack{A \subseteq N\\ |A| \geq 2}} \lambda_A\cdot u_A.
\]
For any $v\in\calG(N)$ and every nonempty $B\subseteq N$ we define a game
\begin{equation}\label{def:vB}
v^B= v(B)\cdot u_B + (v(N)-v(B))\cdot u_N.
\end{equation}
The max-decomposition of an arbitrary coalitional game was proved in \cite{llerenaRafels06}. Herein  a simple proof of the same result is provided in our special setting of zero-normalized games.

\begin{lemma}\label{lem:WSdecomp}
The following hold true for any $v\in\calGZM(N)$:
\begin{enumerate}
\item $v= \bigvee_{\emptyset \neq B\subseteq N} v^B$.
\item $v^B\in\calGTM^0(N)$ and $v^B(N)=v(N)$, for all nonempty $B\subseteq N$.
\end{enumerate}
Conversely, let $v_1,\dots,v_k\in \calGTM^0(N)$ be such that $v_1(N)=\dots =v_k(N)$ and define $w= \bigvee_{i=1}^k v_i$. Then $w\in\calGZM(N)$ and $v_i(N)=w(N)$ for all $i=1,\dots,k$.
\end{lemma}
\begin{proof}
Since $v$ is monotone, both $v(B)$ and $v(N)-v(B)$ are nonnegative for every $\emptyset \neq B\subseteq N$. Hence,  $v^B$ is totally monotone since it is a conic combination of unanimity games $u_B$ and $u_N$. It follows from the definition of $v^B$ that $v^B(N)=v(N)$ and $v^B(B)=v(B)$, for all $\emptyset \neq B\subseteq N$. Hence, we only need to show that $v^B(A)\leq v(A)$ for every nonempty $A\subsetneq N$. Let $A\subseteq B$. Then $v^B(A)=v(B)\leq v(A)$, by monotonicity of $v$. If $A \not\subseteq B$, the the same argument shows that  $v^B(A)=0\leq v(A)$.

For the second part of the assertion, observe that $w(\{i\})=0$ for all $i\in N$. Since every game $w_i$ is monotone, it follows that for all $A\subseteq B\subseteq N$,
\[
w(A)=\bigvee_{i=1}^k v_i(A) \leq \bigvee_{i=1}^k v_i(B)=w(B).
\]
Clearly, $v_i(N)=w(N)$ for all $i=1,\dots,k$. \qed
\end{proof}

Combining Lemma \ref{pro:coreint} with Lemma \ref{lem:WSdecomp}, the core of every game $v\in\calGZM(N)$ can be written as
\begin{equation}\label{eq:core}
\calC(v)=\bigcap_{\emptyset \neq B\subseteq N} \calC(v^B),
\end{equation}
where $v^B$ is defined by \eqref{def:vB}. We will prove that the core solution on $\Gamma(N)=\calGWS(N)$ factors through the set $\Omega(N)^Z$, where we put $\Omega(N)=\calGTM^0(N)\cup \calGA(N)$ and $Z=\frakP(N)$.

\begin{theorem}\label{thm:WSmain}
Define the maps
\[
\tau\colon  \calGWS(N) \to (\calGTM^0(N)\cup \calGA(N))^{\frakP(N)}
\]
and 
\[
\alpha\colon (\calGTM^0(N)\cup \calGA(N))^{\frakP(N)} \to \frakP(\dR^n)
\]
as
$\tau(v)= (v-\hat{v},(\hat{v}^B)_{B\in \frakP(N)\setminus \{\emptyset\}})$, for all $v\in  \calGWS(N)$, and
\[
\alpha(\omega)= \calC(\omega_{\emptyset}) + \bigcap_{\substack{B\in \frakP(N)\\ B\neq \emptyset}} \calC(\omega_B), \qquad \omega=(\omega_B)_{B\in\frakP(N)}\in (\calGTM^0(N)\cup \calGA(N))^{\frakP(N)}.
\]
Then this diagram commutes:
\begin{equation*}\label{diagram3}
   \begin{tikzcd}
  \calGWS(N) \arrow[rd,"\calC"] \arrow[r, "\tau"] & (\calGTM^0(N)\cup \calGA(N))^{\frakP(N)} \arrow[d,"\alpha"] \\
   & \frakP(\dR^n)
     \end{tikzcd}
\end{equation*}
\end{theorem}
\begin{proof}
Let $v\in \calGWS(N)$. Then
\begin{equation}\label{eq:commute}
\alpha(\tau(v))=\alpha((v-\hat{v},(\hat{v}^B)_{B\in \frakP(N)\setminus \{\emptyset\}}))=\calC(v-\hat{v})+ \bigcap_{\substack{B\in \frakP(N)\\ B\neq \emptyset}} \calC(\hat{v}^B).
\end{equation}
As $\hat{v}$ is zero-normalized we get $\hat{v}=\bigvee_{\emptyset \neq B\subseteq N} \hat{v}^B$ and \eqref{eq:core} says that
\begin{equation}\label{eq:int}
\bigcap_{\substack{B\in \frakP(N)\\ B\neq \emptyset}} \calC(\hat{v}^B)=\calC(\hat{v}).
\end{equation}
Since $v-\hat{v}$ is an additive game, the core $\calC(v-\hat{v})$ is a singleton whose only payoff allocation is the vector $e^{-1}(v-\hat{v})$, where $e^{-1}$ is the linear map \eqref{def:addinv}. As the core solution is covariant under strategic equivalence, combining \eqref{eq:commute} with \eqref{eq:int} yields
\[
\alpha(\tau(v))=\calC(v-\hat{v})+\calC(\hat{v})=\calC(v).
\]
Hence, the diagram commutes.
\qed
\end{proof}

Theorem \ref{thm:WSmain} streamlines the decomposition \eqref{eq:core} into the cores of games $v^B$. Such cores have a very special shape from the viewpoint of the theory of polyhedra. Observe that the core of a unanimity game $u_A$ with $\emptyset\neq A\subseteq N$ is the standard $(|A|-1)$-simplex
\[
\Delta_A = \conv \{\delta^i \mid i\in A\},
\]
where a vector $\delta^i\in\dR^n$ has coordinates
\[
\delta_j^i = \begin{cases}
1 & i=j,\\
0 & i\neq j.
\end{cases}
\]
Using (\ref{def:vB}) we get
\begin{align*}
\calC(v^B)&=v(B)\cdot \calC(u_B) + (v(N)-v(B))\cdot \calC(u_N)\\ & =v(B)\cdot \Delta_B + (v(N)-v(B))\cdot \Delta_N \\
& =\conv \{v(B)\cdot \delta^i + (v(N)-v(B))\cdot \delta^j \mid i\in B, j\in N\}.
\end{align*}
Thus, each core $\calC(v^B)$ is a special weighted Minkowski sum of two standard simplices, which are called \emph{nestohedra} and count among important convex polytopes studied in \cite{Postnikov:Faces08}. Then one of the consequences of Theorem \ref{thm:WSmain} is that the core of every weakly superadditive game is (a possibly translated) intersection \eqref{eq:core} of such nestohedra. It is an interesting open question for further research if the intersection \eqref{eq:core} allows for an alternative characterization, which would capture the class of all polytopes arising as cores of weakly superadditive games.

\end{document}